\documentclass[12pt]{amsart}

\usepackage[bookmarks=false,unicode,colorlinks,urlcolor=red,citecolor=red,linkcolor=blue]{hyperref}
\usepackage{verbatim}
\usepackage{fullpage}
\usepackage{amsfonts}
\usepackage{amssymb}

\usepackage{lipsum}
\usepackage{tikz}
\usetikzlibrary{decorations.pathreplacing}
\usetikzlibrary{arrows}

\newcommand{\RP}{\mathbb{R}\mathrm{P}}
\newcommand{\CP}{\mathbb{C}\mathrm{P}}
\renewcommand{\tilde}{\widetilde}
\newcommand{\im}{\hbox{im}}
\newcommand{\R}{\mathbb{R}}
\renewcommand{\C}{\mathbb{C}}
\newcommand{\N}{\mathbb{N}}
\newcommand{\Z}{\mathbb{Z}}
\newcommand{\Q}{\mathbb{Q}}
\renewcommand{\H}{\mathbb{H}}
\newcommand{\K}{\mathbb{K}}
\newcommand{\T}{\mathbb{T}}
\newcommand{\F}{\mathbb{F}}
\newcommand{\Hil}{\mathcal{H}}
\newcommand{\Kil}{\mathcal{K}}
\newcommand{\B}{\mathcal{B}}
\newcommand{\1}{\mathbbm{1}}
\newcommand{\eps}{\varepsilon}
\newcommand{\la}{\lambda}
\newcommand{\ve}{\varepsilon}
\newcommand{\Contradiction}{\Rightarrow\Leftarrow}
\DeclareMathOperator{\lspan}{span}
\DeclareMathOperator{\tr}{tr}
\DeclareMathOperator{\ran}{ran}
\DeclareMathOperator{\supp}{supp}
\DeclareMathOperator{\Supp}{Supp}
\DeclareMathOperator{\diag}{diag}
\providecommand{\abs}[1]{\lvert#1\rvert}
\providecommand{\norm}[1]{\lVert#1\rVert}
\renewcommand{\theenumi}{\roman{enumi}}
\renewcommand{\labelenumi}{(\theenumi)}

\newcounter{Theorem}

\numberwithin{equation}{section}
\numberwithin{Theorem}{section}

\theoremstyle{plain} 
\newtheorem{thm}[Theorem]{Theorem}

\newtheorem{claim}[Theorem]{Claim}

\theoremstyle{definition}
\newtheorem{defn}[Theorem]{Definition}

\theoremstyle{remark}

\newtheorem{ex}[Theorem]{Example}

\newcommand{\mref}[1]{%
\href{http://www.ams.org/mathscinet-getitem?mr=#1}{#1}}

\newcommand{\arxiv}[1]{%
\href{http://front.math.ucdavis.edu/#1}{ArXiv:#1}}

\newcommand{\zbl}[1]{%
\href{http://zbmath.org/?q=an:#1}{Zbl #1}}

\newcommand{\doi}[1]{%
\href{http://dx.doi.org/#1}{doi:#1}}

\begin{document}

\title{A pathological construction for real functions with large collections of level sets}

\author{Gavin Armstrong}
\address{Department of Mathematics, University of Oregon, Eugene, OR 97403--1222, USA}
\email{gka@uoregon.edu}


\date{\today}

\begin{abstract}

Consider all the level sets of a real function. We can group these level sets according to their Hausdorff dimensions. We show that the Hausdorff dimension of the collection of all level sets of a given Hausdorff dimension can be arbitrarily close to 1, even if the function is differentiable to some level. By definition of Hausdorff dimension it is clear, for any real function $f(x)$ and any $\alpha \in [0,1]$, that $\dim_{H} \left\{ \hspace{0.03in} y \ : \ \dim_{H} (f^{-1}(y)) \geq \alpha \hspace{0.03in} \right\} \leq 1$. What is surprising, and what we show, is that this is actually a sharp bound. That is, $$\sup \left\{ \hspace{0.03in} \dim_{H} \left\{ \hspace{0.03in} y \ : \ \dim_{H} (f^{-1}(y)) = 1 \hspace{0.03in} \right\} \ : \ f \in C^{k} \hspace{0.03in} \right\} = 1,$$
for any $k \in \mathbb{Z}_{\geq 0}$.
\end{abstract}

\maketitle

\vspace{-0.2in}

\section{Preliminaries}\label{S1}

For the purposes of this paper it will be sufficient to consider functions of the form $$f : [0,1] \rightarrow [0,1].$$

\*

Let $y \in [0,1]$ and consider the level set $f^{-1}(y) \subseteq [0,1]$.

\*

For any $d \in [0,\infty)$, this level set has a $d$-dimensional {\bf Hausdorff content} given by
$$C^{d}_{H} \left( f^{-1}(y) \right) = \inf \left\{ \hspace{0.03in} \sum_i r_{i}^{d} \ : \ \text{ there is a cover of } f^{-1}(y) \text{ by balls of radii } r_i > 0 \hspace{0.03in} \right\}.$$

Further, $f^{-1}(y)$ has a {\bf Hausdorff dimension} given by
$$\dim_{H} \left( f^{-1}(y) \right) = \inf \left\{ \hspace{0.03in} d \geq 0 \ : \ C_{H}^{d} \left( f^{-1} (y) \right) = 0 \hspace{0.03in} \right\}.$$

\*

We are interested in all those $y$ whose pre-images have positive Hausdorff dimension:
\begin{eqnarray}
\left\{ \hspace{0.03in} y \in [0,1] \ : \ \dim_{H} \left( f^{-1}(y) \right) > 0 \hspace{0.03in} \right\}. \nonumber
\end{eqnarray}

\noindent More specifically though we are interested in the sets $$\left\{ \hspace{0.03in} y \in [0,1] \ : \ \dim_{H} \left( f^{-1}(y) \right) \geq \alpha \hspace{0.03in} \right\},$$
where $0 \leq \alpha \leq 1$. 

We wish to find functions, $f(x)$, that maximize the Hausdorff dimension of this set.

\vfill

\pagebreak

\begin{defn}\label{d1}
Let $0 \leq \alpha \leq 1$. Define
\[
I_{\alpha}(f) = \left\{ \hspace{0.03in} y \in [0,1] \ : \ \dim_{H} \left( f^{-1}(y) \right) \geq \alpha \hspace{0.03in} \right\}.
\]
\end{defn}

\*

\noindent Note: Trivially, for any function $f(x)$, we have $I_{0} (f) = Range (f)$.
\vspace{0.1in}

\vspace{0.5in}

\section{Examples}

\begin{ex}[\textbf{Trivial Example}]

\*

Consider the graph of the function
\begin{alignat*}{4}
f \ :\ [0&,1] &&\quad \longrightarrow &\quad [0&,1] \\
&x &&\quad \longmapsto &\quad -x(&x-1).
\end{alignat*}

\*

\begin{center}
\begin{tikzpicture}[scale=1.2]
      \draw[->] (0,0) -- (6,0)  node[right]{$x$};
      \draw[->] (0.5,-0.5) -- (0.5,5.5) node[above]{$y$};
      
      \foreach \x [evaluate=\x as \xval using 5*\x] in {0.2, 0.4, 0.6, 0.8, 1}{
	  	\draw (\xval + 0.5, -0.1) -- node[below,yshift=-0.02in]{$\x$} (\xval + 0.5, 0.1);
      }
      
      \draw[domain=0.5:5.5,smooth,variable=\x,blue] plot ({\x},{-5*(\x-0.5)*(\x - 5.5)/6.25});

      \draw (4,4.5) node[right,blue]{$f(x) = -x(x-1)$};
      
      \draw[decoration={brace},decorate,thick]  (-1.5,0) -- node[left=6pt] {$I_0 (f)$} (-1.5,5);

      \draw[dashed] (-1.5,0) -- (5.5,0);
      
      \draw[dashed] (-1.5,5) -- (3,5);
      
      \draw[thick] (-0.75,0) -- (-0.75,5);
      
      \draw[thick] (-0.65,0) -- (-0.85,0);
      \draw[thick] (-0.65,5) -- (-0.85,5);
      
      \foreach \y [evaluate=\y as \yval using 20*\y] in {0.05, 0.1, 0.15, 0.2, 0.25}{
      	\draw[rounded corners, white, fill=white] (-0.25, \yval - 0.25) rectangle (0.42, \yval + 0.25);
      	\draw (0.5 - 0.1,\yval) -- node[left]{$\y$} (0.5 + 0.1,\yval);
      }

\end{tikzpicture}
\end{center}

\*

As expected, in this case $I_{0}(f) = [0,0.25]$. 

\*

Note that the pre-image of each point in the range of $f(x)$ is at most finite. Thus the pre-image of each point has trivial Hausdorff dimension. Hence $$I_{\alpha}(f) =\emptyset \hspace{0.5in} \text{ and } \hspace{0.5in} \dim_{H} I_{\alpha} (f) = 0,$$ for all $\alpha > 0$.

\end{ex}

\vfill

\pagebreak

\begin{ex}[\textbf{Another Trivial Example}]

\*

Consider any constant function. For example:
\[
f \ : \ [0,1] \ \longrightarrow \ [0,1], \quad\quad x \ \longmapsto 0.5.
\]

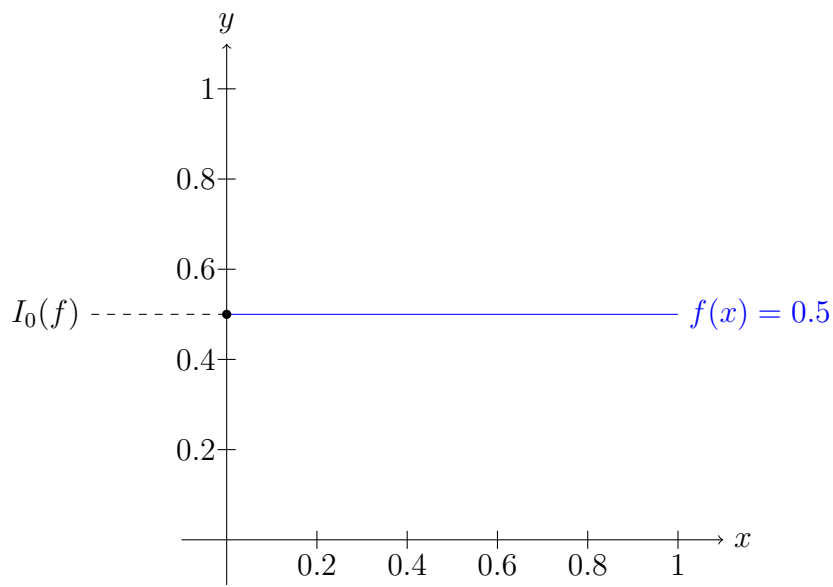
\begin{center}
\begin{tikzpicture}[scale=1.2]
      \draw[->] (0,0) -- (6,0)  node[right]{$x$};
      \draw[->] (0.5,-0.5) -- (0.5,5.5) node[above]{$y$};
      
      \foreach \x [evaluate=\x as \xval using 5*\x] in {0.2, 0.4, 0.6, 0.8, 1}{
	  	\draw (\xval + 0.5, -0.1) -- node[below,yshift=-0.02in]{$\x$} (\xval + 0.5, 0.1);
      }
      
      \foreach \y [evaluate=\y as \yval using 5*\y] in {0.2, 0.4, 0.6, 0.8, 1}{
	  	\draw (0.5 - 0.1,\yval) -- node[left]{$\y$} (0.5 + 0.1,\yval);
      }
      
      \draw[domain=0.5:5.5,smooth,variable=\x,blue] plot ({\x},{2.5});

      \draw (5.5,2.5) node[right,blue]{$f(x) = 0.5$};

      \draw[dashed] (-1,2.5) -- (0.5,2.5);

      \draw (0.5,2.5) node[circle,fill, inner sep=0.01pt] {.};
      \draw (-1,2.5) node[left] {$I_0 (f)$};

\end{tikzpicture}
\end{center}

\*

In this case the only non-trivial pre-image is $f^{-1}\left( 0.5 \right) = [0,1]$.

\*

The unit interval has Hausdorff dimension $1$, and so $$I_{\alpha}(f) = \left\{ 0.5 \right\} \hspace{0.5in} \text{ and } \hspace{0.5in} \dim_{H}  I_{\alpha} (f) = 0,$$ for all $0 \leq \alpha \leq 1$.

\end{ex}

\vfill

\noindent The next question is: How large can we make $I_{\alpha} (f)$, for $\alpha > 0$, while preserving continuity or even differentiability? 

\*

The next example shows that we can construct a continuous function $f(x)$ such that $I_{1}(f)$ is infinite.

\vfill

\pagebreak

\begin{ex}(\textbf{Non-Trivial $I_{\alpha}(f)$})

Consider the function
\[
f_1 \ : \ \left[ 0, 0.5 \right] \ \longrightarrow \ [0,1], \quad \quad x \ \longmapsto \begin{cases} \frac{3}{2}x & \text{ if } x \in \left[0, \frac{1}{6} \right] \\
\frac{1}{4} & \text{ if } x \in \left[ \frac{1}{6}, \frac{1}{3} \right] \\
\frac{3}{2}x - \frac{1}{4} & \text{ if } x \in \left[ \frac{1}{3}, \frac{1}{2} \right] \end{cases}
\]

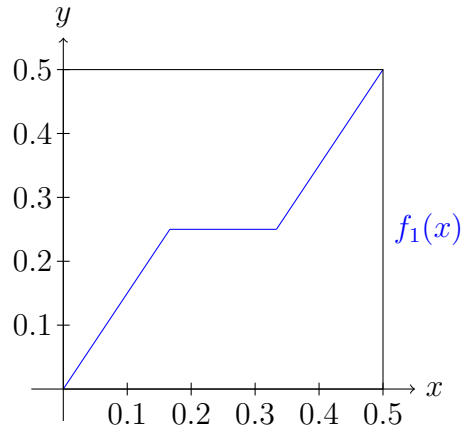
\begin{center}
\begin{tikzpicture}[scale=0.85]
      \draw[->] (0,0) -- (6,0)  node[right]{$x$};
      \draw[->] (0.5,-0.5) -- (0.5,5.5) node[above]{$y$};
      
      \foreach \x [evaluate=\x as \xval using 10*\x] in {0.1, 0.2, 0.3, 0.4, 0.5}{
	  	\draw (\xval + 0.5, -0.1) -- node[below,yshift=-0.02in]{$\x$} (\xval + 0.5, 0.1);
      }
      
      \foreach \y [evaluate=\y as \yval using 10*\y] in {0.1, 0.2, 0.3, 0.4, 0.5}{
	  	\draw (0.5 - 0.1,\yval) -- node[left]{$\y$} (0.5 + 0.1,\yval);
      }
      
      \draw[domain=0.5:2.16667,smooth,variable=\x,blue] plot ({\x},{1.5*(\x - 0.5)});

      \draw[domain=2.16667:3.8333,smooth,variable=\x,blue] plot ({\x},{2.5});
      
      \draw[domain=3.8333:5.5,smooth,variable=\x,blue] plot ({\x},{1.5*(\x - 0.5) - 2.5});
      
      \draw (0.5,5) -- (5.5,5) -- (5.5,0) -- (0.5,0) -- (0.5,5) -- cycle;
      
      \draw (5.5,2.5) node[right,blue]{$f_1 (x)$};

\end{tikzpicture}
\end{center}

\*

Take this function and make scaled copies of it with dimensions $\frac{1}{2^k} \times \frac{1}{2^k}$. Then graph these scaled functions end-to-end so that the bottom left coordinate of the $k$-th graph coincides with the point $\left( 1 - \frac{1}{2^{k-1}}, 1 - \frac{1}{2^{k-1}} \right)$.

\*

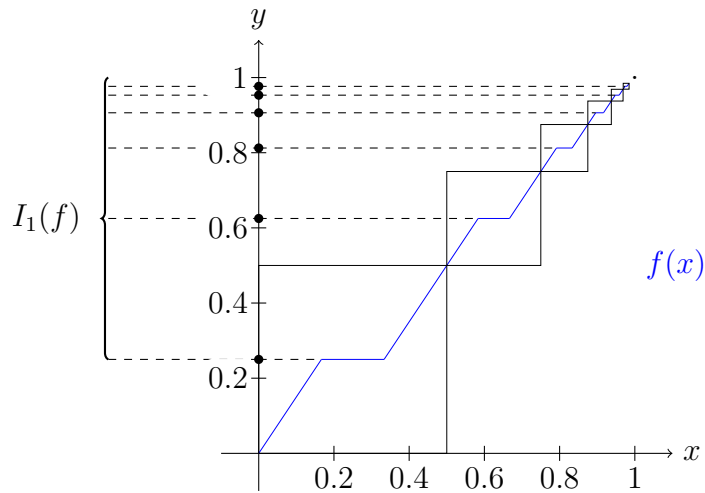
\begin{center}
\begin{tikzpicture}[scale=1]
      \draw[->] (0,0) -- (6,0)  node[right]{$x$};
      \draw[->] (0.5,-0.5) -- (0.5,5.5) node[above]{$y$};
      
      \foreach \x [evaluate=\x as \xval using 5*\x] in {0.2, 0.4, 0.6, 0.8, 1}{
	  	\draw (\xval + 0.5, -0.1) -- node[below,yshift=-0.02in]{$\x$} (\xval + 0.5, 0.1);
      }
      
      \foreach \t [evaluate=\t as \txval using 0.5*2^(\t) + 5*(1 - 1/2^(\t - 1))*2^(\t), evaluate=\t as \tyval using 5*(1 - 1/2^(\t - 1))*2^(\t)] in {1, 2, 3, 4, 5, 6}{
	      \begin{scope}[shift={(\txval, \tyval)}, transform canvas={scale=1/2^(\t)}]
      		  \draw[domain=0:1.66667,smooth,variable=\x,blue] plot ({\x},{1.5*(\x)});

	     	  \draw[domain=1.66667:3.3333,smooth,variable=\x,blue] plot ({\x},{2.5});
      
	      	  \draw[domain=3.3333:5,smooth,variable=\x,blue] plot ({\x},{1.5*(\x) - 2.5});
      
	      	  \draw (0,5) -- (5,5) -- (5,0) -- (0,0) -- (0,5) -- cycle;
	      \end{scope};
	}

      \draw (5.5,2.5) node[right,blue]{$f (x)$};
      
      \draw (5.5,5) node {.};
      
      \foreach \s [evaluate=\s as \syval using 5/2^(\s + 1) + 5*(1 - 1/2^(\s - 1)), evaluate=\s as \sxval using 0.5 + 5*(1 - 1/2^(\s- 1)) + 5*0.333/2^(\s)] in {1, ..., 6}{
	      \draw (0.5, \syval) node[circle, fill, inner sep=0.01pt] {.};
	      \draw[dashed] (-1.5, \syval) -- (0.5, \syval);
	      \draw[dashed] (0.5, \syval) -- (\sxval,\syval);
      }

      \draw[decoration={brace},decorate,thick]  (-1.5,5/4) -- node[left=6pt] {$I_1 (f)$} (-1.5,5);
      
      \foreach \y [evaluate=\y as \yval using 5*\y] in {0.2, 0.4, 0.6, 0.8, 1}{
      		\draw[rounded corners, white, fill=white] (-0.25, \yval - 0.25) rectangle (0.42, \yval + 0.25);
	  	\draw (0.5 - 0.1,\yval) -- node[left]{$\y$} (0.5 + 0.1,\yval);
      }

\end{tikzpicture}
\end{center}

\*

This gives us a continuous (although not differentiable) function $f : [0,1] \longrightarrow [0,1]$ such that 
$$I_{\alpha}(f) = \left\{ \frac{1}{4}, \frac{5}{8}, \hdots, \frac{2^{i} - \frac{3}{2}}{2^{i}}, \hdots \right\} \hspace{0.4in} \text{ and } \hspace{0.4in} \dim_{H}  I_{\alpha} (f) = 0,$$ 
for all $0 < \alpha \leq 1$.
\end{ex}

\vfill

\pagebreak

\section{Main Theorem}

\noindent In this paper we show the following very counterintuitive result: 

We can make $\dim_H I_{\alpha}(f) \leq 1$ arbitrarily close to $1$, for all $0 \leq \alpha \leq 1$, while still maintaining the continuity and even differentiability of $f(x)$.

\begin{thm}\label{t1}
For any $k \in \mathbb{Z}_{\geq 0}$ and any $0 \leq \alpha \leq 1$ we have
\[
\sup \left\{ \hspace{0.03in} \dim_{H} (I_{\alpha} (f) ) \ : \ f \in C^{k} \hspace{0.03in} \right\} = 1. \nonumber
\]
\end{thm}

\*

\begin{ex}(\textbf{Main Function})
Consider the following iteratively defined function.

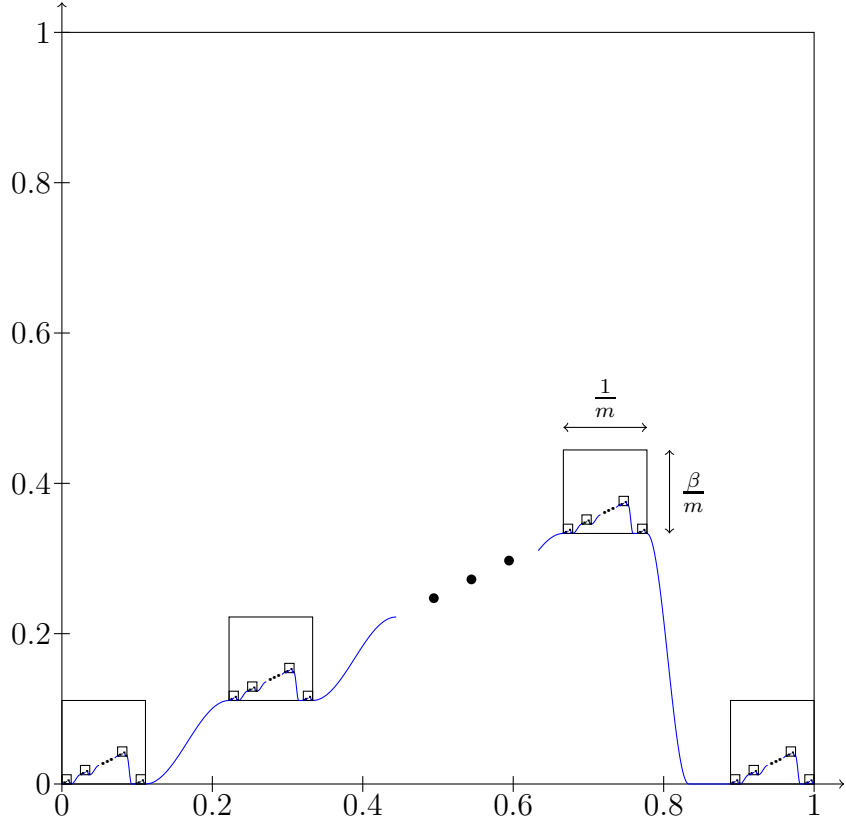
\begin{center}
\begin{tikzpicture}[scale=1]
\draw[->] (0,0) -- (0,10.4);
\draw[->] (0,0) -- (10.4,0);


\draw (10,0) -- (10,10) -- (0,10);

\foreach \x in {0, 0.2, 0.4, 0.6, 0.8, 1}
\draw[shift={(10*\x,0)},color=black] (0pt,3pt) -- (0pt,-3pt);
\foreach \x in {0, 0.2, 0.4, 0.6, 0.8, 1}
\draw[shift={(10*\x,0)},color=black] (0pt,0pt) node[below] {$\x$};

\foreach \y in {0, 0.2, 0.4, 0.6, 0.8, 1}
\draw[shift={(0, 10*\y)},color=black] (3pt, 0pt) -- (-3pt, 0pt);
\foreach \y in {0, 0.2, 0.4, 0.6, 0.8, 1}
\draw[shift={(0, 10*\y)},color=black] (0pt,0pt) node[left] {$\y$};

\draw[<->] (7*10/9 + 0.3, 3*10/9) -- node[right]{$\frac{\beta}{m}$} ++ (0, 1*10/9); 

\draw[<->] (6*10/9, 4*10/9 + 0.3) -- node[above]{$\frac{1}{m}$} ++ (1*10/9, 0);

\draw (0,0) -- (10/9,0) -- (10/9,10/9) -- (0,10/9) -- (0,0);
\draw (0 + 2*10/9, 0 + 10/9) -- (10/9 + 2*10/9, 0 + 10/9) -- (10/9 + 2*10/9, 10/9 + 10/9) -- (0 + 2*10/9, 10/9 + 10/9)  -- (0 + 2*10/9, 0 + 10/9);
\draw (0 + 6*10/9, 0 + 3*10/9) -- (10/9 + 6*10/9, 0 + 3*10/9) -- (10/9 + 6*10/9, 10/9 + 3*10/9) -- (0 + 6*10/9, 10/9 + 3*10/9)  -- (0 + 6*10/9, 0 + 3*10/9);

\draw (0 + 8*10/9, 0 + 0*10/9) -- (10/9 + 8*10/9, 0 + 0*10/9) -- (10/9 + 8*10/9, 10/9 + 0*10/9) -- (0 + 8*10/9, 10/9 + 0*10/9)  -- (0 + 8*10/9, 0 + 0*10/9);

\draw (4*10/9 + 1/2, 2*10/9 + 1/4) node[circle,fill, inner sep=0.1pt] {.};
\draw (4*10/9 + 2*1/2, 2*10/9 + 2*1/4) node[circle,fill, inner sep=0.1pt] {.};
\draw (4*10/9 + 3*1/2, 2*10/9 + 3*1/4) node[circle,fill, inner sep=0.1pt] {.};

\draw[domain=10/9:2*10/9,smooth,variable=\x,blue] plot ({\x},{(10/9)*(1 + sin(pi*(\x*(9/10) - 1 - 1/2) r))/2});

\draw[domain=3*10/9:4*10/9,smooth,variable=\x,blue] plot ({\x},{(10/9)*(1 + sin(pi*(\x*(9/10) - (3 + 1/2)) r))/2 + 1*10/9});

\draw[domain=5.7*10/9:6*10/9,smooth,variable=\x,blue] plot ({\x},{(10/9)*(1 + sin(pi*(\x*(9/10) - (5 + 1/2)) r))/2 + 2*10/9});

\draw[domain=7*10/9:7.5*10/9,smooth,variable=\x,blue] plot ({\x},{(3*10/9)*(1 + sin(2*pi*(\x*(9/10) + 1/4 - 7) r))/2 + 0*10/9});

\draw[domain=7.5*10/9:8*10/9,smooth,variable=\x,blue] plot ({\x},{0});

\draw (0,0) -- (10/81,0) -- (10/81,10/81) -- (0,10/81) -- (0,0);

\draw (0 + 2*10/81, 0 + 10/81) -- (10/81 + 2*10/81, 0 + 10/81) -- (10/81 + 2*10/81, 10/81 + 10/81) -- (0 + 2*10/81, 10/81 + 10/81)  -- (0 + 2*10/81, 0 + 10/81);
\draw (0 + 6*10/81, 0 + 3*10/81) -- (10/81 + 6*10/81, 0 + 3*10/81) -- (10/81 + 6*10/81, 10/81 + 3*10/81) -- (0 + 6*10/81, 10/81 + 3*10/81)  -- (0 + 6*10/81, 0 + 3*10/81);

\draw (0 + 8*10/81, 0 + 0*10/81) -- (10/81 + 8*10/81, 0 + 0*10/81) -- (10/81 + 8*10/81, 10/81 + 0*10/81) -- (0 + 8*10/81, 10/81 + 0*10/81)  -- (0 + 8*10/81, 0 + 0*10/81);

\draw (4*10/81 + 1/18, 2*10/81 + 1/36) node {.};
\draw (4*10/81 + 2*1/18, 2*10/81 + 2*1/36) node[circle] {.};
\draw (4*10/81 + 3*1/18, 2*10/81 + 3*1/36) node[circle] {.};

\draw[domain=10/81:2*10/81,smooth,variable=\x,blue] plot ({\x},{(10/81)*(1 + sin(pi*(\x*(81/10) - 1 - 1/2) r))/2});

\draw[domain=3*10/81:4*10/81,smooth,variable=\x,blue] plot ({\x},{(10/81)*(1 + sin(pi*(\x*(81/10) - (3 + 1/2)) r))/2 + 1*10/81});

\draw[domain=5.7*10/81:6*10/81,smooth,variable=\x,blue] plot ({\x},{(10/81)*(1 + sin(pi*(\x*(81/10) - (5 + 1/2)) r))/2 + 2*10/81});

\draw[domain=7*10/81:7.5*10/81,smooth,variable=\x,blue] plot ({\x},{(3*10/81)*(1 + sin(2*pi*(\x*(81/10) + 1/4 - 7) r))/2 + 0*10/81});

\draw[domain=7.5*10/81:8*10/81,smooth,variable=\x,blue] plot ({\x},{0});

\draw (0 + 2*10/9,0 + 10/9) -- (10/81 + 2*10/9, 0 + 10/9) -- (10/81 + 2*10/9,10/81 + 10/9) -- (0 + 2*10/9,10/81 + 10/9) -- (0 + 2*10/9,0 + 10/9);
\draw (0 + 2*10/81 + 2*10/9, 0 + 10/81 + 10/9) -- (10/81 + 2*10/81 + 2*10/9, 0 + 10/81 + 10/9) -- (10/81 + 2*10/81 + 2*10/9, 10/81 + 10/81 + 10/9) -- (0 + 2*10/81 + 2*10/9, 10/81 + 10/81 + 10/9)  -- (0 + 2*10/81 + 2*10/9, 0 + 10/81 + 10/9);
\draw (0 + 6*10/81 + 2*10/9, 0 + 3*10/81 + 10/9) -- (10/81 + 6*10/81 + 2*10/9, 0 + 3*10/81 + 10/9) -- (10/81 + 6*10/81 + 2*10/9, 10/81 + 3*10/81 + 10/9) -- (0 + 6*10/81 + 2*10/9, 10/81 + 3*10/81 + 10/9)  -- (0 + 6*10/81 + 2*10/9, 0 + 3*10/81 + 10/9);

\draw (0 + 8*10/81 + 2*10/9, 0 + 0*10/81 + 10/9) -- (10/81 + 8*10/81 + 2*10/9, 0 + 0*10/81 + 10/9) -- (10/81 + 8*10/81 + 2*10/9, 10/81 + 0*10/81 + 10/9) -- (0 + 8*10/81 + 2*10/9, 10/81 + 0*10/81 + 10/9)  -- (0 + 8*10/81 + 2*10/9, 0 + 0*10/81 + 10/9);

\draw (4*10/81 + 1/18 + 2*10/9, 2*10/81 + 1/36 + 10/9) node {.};
\draw (4*10/81 + 2*1/18 + 2*10/9, 2*10/81 + 2*1/36 + 10/9) node[circle] {.};
\draw (4*10/81 + 3*1/18 + 2*10/9, 2*10/81 + 3*1/36 + 10/9) node[circle] {.};

\draw[domain=(10/81 + 2*10/9):(2*10/81 + 2*10/9),smooth,variable=\x,blue] plot ({\x},{(10/81)*(1 + sin(pi*(\x*(81/10) - 1 - 1/2) r))/2 + 10/9});

\draw[domain=(3*10/81 + 2*10/9):(4*10/81 + 2*10/9),smooth,variable=\x,blue] plot ({\x},{(10/81)*(1 + sin(pi*(\x*(81/10) - (3 + 1/2)) r))/2 + 1*10/81 + 10/9});

\draw[domain=(5.7*10/81 + 2*10/9):(6*10/81 + 2*10/9),smooth,variable=\x,blue] plot ({\x},{(10/81)*(1 + sin(pi*(\x*(81/10) - (5 + 1/2)) r))/2 + 2*10/81 + 10/9});

\draw[domain=(7*10/81 + 2*10/9):(7.5*10/81 + 2*10/9),smooth,variable=\x,blue] plot ({\x},{(3*10/81)*(1 + sin(2*pi*(\x*(81/10) + 1/4 - 7) r))/2 + 0*10/81 + 10/9});

\draw[domain=(7.5*10/81 + 2*10/9):(8*10/81 + 2*10/9),smooth,variable=\x,blue] plot ({\x},{0 + 10/9});

\draw (0 + 6*10/9,0 + 3*10/9) -- (10/81 + 6*10/9, 0 + 3*10/9) -- (10/81 + 6*10/9,10/81 + 3*10/9) -- (0 + 6*10/9, 10/81 + 3*10/9) -- (0 + 6*10/9,0 + 3*10/9);
\draw (0 + 2*10/81 + 6*10/9, 0 + 10/81 + 3*10/9) -- (10/81 + 2*10/81 + 6*10/9, 0 + 10/81 + 3*10/9) -- (10/81 + 2*10/81 + 6*10/9, 10/81 + 10/81 + 3*10/9) -- (0 + 2*10/81 + 6*10/9, 10/81 + 10/81 + 3*10/9)  -- (0 + 2*10/81 + 6*10/9, 0 + 10/81 + 3*10/9);
\draw (0 + 6*10/81 + 6*10/9, 0 + 3*10/81 + 3*10/9) -- (10/81 + 6*10/81 + 6*10/9, 0 + 3*10/81 + 3*10/9) -- (10/81 + 6*10/81 + 6*10/9, 10/81 + 3*10/81 + 3*10/9) -- (0 + 6*10/81 + 6*10/9, 10/81 + 3*10/81 + 3*10/9)  -- (0 + 6*10/81 + 6*10/9, 0 + 3*10/81 + 3*10/9);

\draw (0 + 8*10/81 + 6*10/9, 0 + 0*10/81 + 3*10/9) -- (10/81 + 8*10/81 + 6*10/9, 0 + 0*10/81 + 3*10/9) -- (10/81 + 8*10/81 + 6*10/9, 10/81 + 0*10/81 + 3*10/9) -- (0 + 8*10/81 + 6*10/9, 10/81 + 0*10/81 + 3*10/9)  -- (0 + 8*10/81 + 6*10/9, 0 + 0*10/81 + 3*10/9);

\draw (4*10/81 + 1/18 + 6*10/9, 2*10/81 + 1/36 + 3*10/9) node {.};
\draw (4*10/81 + 2*1/18 + 6*10/9, 2*10/81 + 2*1/36 + 3*10/9) node[circle] {.};
\draw (4*10/81 + 3*1/18 + 6*10/9, 2*10/81 + 3*1/36 + 3*10/9) node[circle] {.};

\draw[domain=(10/81 + 6*10/9):(2*10/81 + 6*10/9),smooth,variable=\x,blue] plot ({\x},{(10/81)*(1 + sin(pi*(\x*(81/10) - 1 - 1/2) r))/2 + 3*10/9});

\draw[domain=(3*10/81 + 6*10/9):(4*10/81 + 6*10/9),smooth,variable=\x,blue] plot ({\x},{(10/81)*(1 + sin(pi*(\x*(81/10) - (3 + 1/2)) r))/2 + 1*10/81 + 3*10/9});

\draw[domain=(5.7*10/81 + 6*10/9):(6*10/81 + 6*10/9),smooth,variable=\x,blue] plot ({\x},{(10/81)*(1 + sin(pi*(\x*(81/10) - (5 + 1/2)) r))/2 + 2*10/81 + 30/9});

\begin{scope}[shift={(6*10/9,0)}]
\draw[domain=(7*10/81 + 0*10/9):(7.5*10/81 + 0*10/9), smooth, variable=\x,blue] plot ({\x}, {3*(10/81)*(1/2)*(1 + cos(2*pi*\x*(81/10) r)) + 3*10/9});
\end{scope}

\draw[domain=(7.5*10/81 + 6*10/9):(8*10/81 + 6*10/9),smooth,variable=\x,blue] plot ({\x},{0 + 30/9});

\begin{scope}[shift={(8*10/9,0)}]
\draw (0,0) -- (10/81,0) -- (10/81,10/81) -- (0,10/81) -- (0,0);
\draw (0 + 2*10/81, 0 + 10/81) -- (10/81 + 2*10/81, 0 + 10/81) -- (10/81 + 2*10/81, 10/81 + 10/81) -- (0 + 2*10/81, 10/81 + 10/81)  -- (0 + 2*10/81, 0 + 10/81);
\draw (0 + 6*10/81, 0 + 3*10/81) -- (10/81 + 6*10/81, 0 + 3*10/81) -- (10/81 + 6*10/81, 10/81 + 3*10/81) -- (0 + 6*10/81, 10/81 + 3*10/81)  -- (0 + 6*10/81, 0 + 3*10/81);

\draw (0 + 8*10/81, 0 + 0*10/81) -- (10/81 + 8*10/81, 0 + 0*10/81) -- (10/81 + 8*10/81, 10/81 + 0*10/81) -- (0 + 8*10/81, 10/81 + 0*10/81)  -- (0 + 8*10/81, 0 + 0*10/81);

\draw (4*10/81 + 1/18, 2*10/81 + 1/36) node {.};
\draw (4*10/81 + 2*1/18, 2*10/81 + 2*1/36) node[circle] {.};
\draw (4*10/81 + 3*1/18, 2*10/81 + 3*1/36) node[circle] {.};

\draw[domain=10/81:2*10/81,smooth,variable=\x,blue] plot ({\x},{(10/81)*(1 + sin(pi*(\x*(81/10) - 1 - 1/2) r))/2});

\draw[domain=3*10/81:4*10/81,smooth,variable=\x,blue] plot ({\x},{(10/81)*(1 + sin(pi*(\x*(81/10) - (3 + 1/2)) r))/2 + 1*10/81});

\draw[domain=5.7*10/81:6*10/81,smooth,variable=\x,blue] plot ({\x},{(10/81)*(1 + sin(pi*(\x*(81/10) - (5 + 1/2)) r))/2 + 2*10/81});

\draw[domain=7*10/81:7.5*10/81,smooth,variable=\x,blue] plot ({\x},{(3*10/81)*(1 + sin(2*pi*(\x*(81/10) + 1/4 - 7) r))/2 + 0*10/81});

\draw[domain=7.5*10/81:8*10/81,smooth,variable=\x,blue] plot ({\x},{0});
\end{scope}

\draw (0,0) -- (10/729,0) -- (10/729, 10/729) -- (0,10/729) -- (0,0);
\draw (0 + 2*10/729, 0 + 10/729) -- (10/729 + 2*10/729, 0 + 10/729) -- (10/729 + 2*10/729, 10/729 + 10/729) -- (0 + 2*10/729, 10/729 + 10/729)  -- (0 + 2*10/729, 0 + 10/729);
\draw (0 + 6*10/729, 0 + 3*10/729) -- (10/729 + 6*10/729, 0 + 3*10/729) -- (10/729 + 6*10/729, 10/729 + 3*10/729) -- (0 + 6*10/729, 10/729 + 3*10/729)  -- (0 + 6*10/729, 0 + 3*10/729);

\draw (0 + 8*10/729, 0 + 0*10/729) -- (10/729 + 8*10/729, 0 + 0*10/729) -- (10/729 + 8*10/729, 10/729 + 0*10/729) -- (0 + 8*10/729, 10/729 + 0*10/729)  -- (0 + 8*10/729, 0 + 0*10/729);

\draw[loosely dotted] (4*10/729 + 0/162, 2*10/729 + 0/324) -- (4*10/729 + 5/162, 2*10/729 + 5/324);

\draw[domain=10/729:2*10/729,smooth,variable=\x,blue] plot ({\x},{(10/729)*(1 + sin(pi*(\x*(729/10) - 1 - 1/2) r))/2});

\draw[domain=3*10/729:4*10/729,smooth,variable=\x,blue] plot ({\x},{(10/729)*(1 + sin(pi*(\x*(729/10) - (3 + 1/2)) r))/2 + 1*10/729});

\draw[domain=5.7*10/729:6*10/729,smooth,variable=\x,blue] plot ({\x},{(10/729)*(1 + sin(pi*(\x*(729/10) - (5 + 1/2)) r))/2 + 2*10/729});

\draw[domain=7*10/729:7.5*10/729,smooth,variable=\x,blue] plot ({\x},{(3*10/729)*(1 + sin(2*pi*(\x*(729/10) + 1/4 - 7) r))/2 + 0*10/729});

\draw[domain=7.5*10/729:8*10/729,smooth,variable=\x,blue] plot ({\x},{0});

\begin{scope}[shift={(2*10/81,1*10/81)}]
\draw (0,0) -- (10/729,0) -- (10/729, 10/729) -- (0,10/729) -- (0,0);
\draw (0 + 2*10/729, 0 + 10/729) -- (10/729 + 2*10/729, 0 + 10/729) -- (10/729 + 2*10/729, 10/729 + 10/729) -- (0 + 2*10/729, 10/729 + 10/729)  -- (0 + 2*10/729, 0 + 10/729);
\draw (0 + 6*10/729, 0 + 3*10/729) -- (10/729 + 6*10/729, 0 + 3*10/729) -- (10/729 + 6*10/729, 10/729 + 3*10/729) -- (0 + 6*10/729, 10/729 + 3*10/729)  -- (0 + 6*10/729, 0 + 3*10/729);

\draw (0 + 8*10/729, 0 + 0*10/729) -- (10/729 + 8*10/729, 0 + 0*10/729) -- (10/729 + 8*10/729, 10/729 + 0*10/729) -- (0 + 8*10/729, 10/729 + 0*10/729)  -- (0 + 8*10/729, 0 + 0*10/729);

\draw[loosely dotted] (4*10/729 + 0/162, 2*10/729 + 0/324) -- (4*10/729 + 5/162, 2*10/729 + 5/324);

\draw[domain=10/729:2*10/729,smooth,variable=\x,blue] plot ({\x},{(10/729)*(1 + sin(pi*(\x*(729/10) - 1 - 1/2) r))/2});

\draw[domain=3*10/729:4*10/729,smooth,variable=\x,blue] plot ({\x},{(10/729)*(1 + sin(pi*(\x*(729/10) - (3 + 1/2)) r))/2 + 1*10/729});

\draw[domain=5.7*10/729:6*10/729,smooth,variable=\x,blue] plot ({\x},{(10/729)*(1 + sin(pi*(\x*(729/10) - (5 + 1/2)) r))/2 + 2*10/729});

\draw[domain=7*10/729:7.5*10/729,smooth,variable=\x,blue] plot ({\x},{(3*10/729)*(1 + sin(2*pi*(\x*(729/10) + 1/4 - 7) r))/2 + 0*10/729});

\draw[domain=7.5*10/729:8*10/729,smooth,variable=\x,blue] plot ({\x},{0});

\end{scope}

\begin{scope}[shift={(6*10/81,3*10/81)}]
\draw (0,0) -- (10/729,0) -- (10/729, 10/729) -- (0,10/729) -- (0,0);
\draw (0 + 2*10/729, 0 + 10/729) -- (10/729 + 2*10/729, 0 + 10/729) -- (10/729 + 2*10/729, 10/729 + 10/729) -- (0 + 2*10/729, 10/729 + 10/729)  -- (0 + 2*10/729, 0 + 10/729);
\draw (0 + 6*10/729, 0 + 3*10/729) -- (10/729 + 6*10/729, 0 + 3*10/729) -- (10/729 + 6*10/729, 10/729 + 3*10/729) -- (0 + 6*10/729, 10/729 + 3*10/729)  -- (0 + 6*10/729, 0 + 3*10/729);

\draw (0 + 8*10/729, 0 + 0*10/729) -- (10/729 + 8*10/729, 0 + 0*10/729) -- (10/729 + 8*10/729, 10/729 + 0*10/729) -- (0 + 8*10/729, 10/729 + 0*10/729)  -- (0 + 8*10/729, 0 + 0*10/729);

\draw[loosely dotted] (4*10/729 + 0/162, 2*10/729 + 0/324) -- (4*10/729 + 5/162, 2*10/729 + 5/324);

\draw[domain=10/729:2*10/729,smooth,variable=\x,blue] plot ({\x},{(10/729)*(1 + sin(pi*(\x*(729/10) - 1 - 1/2) r))/2});

\draw[domain=3*10/729:4*10/729,smooth,variable=\x,blue] plot ({\x},{(10/729)*(1 + sin(pi*(\x*(729/10) - (3 + 1/2)) r))/2 + 1*10/729});

\draw[domain=5.7*10/729:6*10/729,smooth,variable=\x,blue] plot ({\x},{(10/729)*(1 + sin(pi*(\x*(729/10) - (5 + 1/2)) r))/2 + 2*10/729});

\draw[domain=7*10/729:7.5*10/729,smooth,variable=\x,blue] plot ({\x},{(3*10/729)*(1 + sin(2*pi*(\x*(729/10) + 1/4 - 7) r))/2 + 0*10/729});

\draw[domain=7.5*10/729:8*10/729,smooth,variable=\x,blue] plot ({\x},{0});
\end{scope}

\begin{scope}[shift={(8*10/81,0*10/81)}]
\draw (0,0) -- (10/729,0) -- (10/729, 10/729) -- (0,10/729) -- (0,0);
\draw (0 + 2*10/729, 0 + 10/729) -- (10/729 + 2*10/729, 0 + 10/729) -- (10/729 + 2*10/729, 10/729 + 10/729) -- (0 + 2*10/729, 10/729 + 10/729)  -- (0 + 2*10/729, 0 + 10/729);
\draw (0 + 6*10/729, 0 + 3*10/729) -- (10/729 + 6*10/729, 0 + 3*10/729) -- (10/729 + 6*10/729, 10/729 + 3*10/729) -- (0 + 6*10/729, 10/729 + 3*10/729)  -- (0 + 6*10/729, 0 + 3*10/729);

\draw (0 + 8*10/729, 0 + 0*10/729) -- (10/729 + 8*10/729, 0 + 0*10/729) -- (10/729 + 8*10/729, 10/729 + 0*10/729) -- (0 + 8*10/729, 10/729 + 0*10/729)  -- (0 + 8*10/729, 0 + 0*10/729);

\draw[loosely dotted] (4*10/729 + 0/162, 2*10/729 + 0/324) -- (4*10/729 + 5/162, 2*10/729 + 5/324);

\draw[domain=10/729:2*10/729,smooth,variable=\x,blue] plot ({\x},{(10/729)*(1 + sin(pi*(\x*(729/10) - 1 - 1/2) r))/2});

\draw[domain=3*10/729:4*10/729,smooth,variable=\x,blue] plot ({\x},{(10/729)*(1 + sin(pi*(\x*(729/10) - (3 + 1/2)) r))/2 + 1*10/729});

\draw[domain=5.7*10/729:6*10/729,smooth,variable=\x,blue] plot ({\x},{(10/729)*(1 + sin(pi*(\x*(729/10) - (5 + 1/2)) r))/2 + 2*10/729});

\draw[domain=7*10/729:7.5*10/729,smooth,variable=\x,blue] plot ({\x},{(3*10/729)*(1 + sin(2*pi*(\x*(729/10) + 1/4 - 7) r))/2 + 0*10/729});

\draw[domain=7.5*10/729:8*10/729,smooth,variable=\x,blue] plot ({\x},{0});
\end{scope}

\begin{scope}[shift={(2*10/9, 1*10/9)}]
\draw (0,0) -- (10/729,0) -- (10/729, 10/729) -- (0,10/729) -- (0,0);
\draw (0 + 2*10/729, 0 + 10/729) -- (10/729 + 2*10/729, 0 + 10/729) -- (10/729 + 2*10/729, 10/729 + 10/729) -- (0 + 2*10/729, 10/729 + 10/729)  -- (0 + 2*10/729, 0 + 10/729);
\draw (0 + 6*10/729, 0 + 3*10/729) -- (10/729 + 6*10/729, 0 + 3*10/729) -- (10/729 + 6*10/729, 10/729 + 3*10/729) -- (0 + 6*10/729, 10/729 + 3*10/729)  -- (0 + 6*10/729, 0 + 3*10/729);

\draw (0 + 8*10/729, 0 + 0*10/729) -- (10/729 + 8*10/729, 0 + 0*10/729) -- (10/729 + 8*10/729, 10/729 + 0*10/729) -- (0 + 8*10/729, 10/729 + 0*10/729)  -- (0 + 8*10/729, 0 + 0*10/729);

\draw[loosely dotted] (4*10/729 + 0/162, 2*10/729 + 0/324) -- (4*10/729 + 5/162, 2*10/729 + 5/324);

\draw[domain=10/729:2*10/729,smooth,variable=\x,blue] plot ({\x},{(10/729)*(1 + sin(pi*(\x*(729/10) - 1 - 1/2) r))/2});

\draw[domain=3*10/729:4*10/729,smooth,variable=\x,blue] plot ({\x},{(10/729)*(1 + sin(pi*(\x*(729/10) - (3 + 1/2)) r))/2 + 1*10/729});

\draw[domain=5.7*10/729:6*10/729,smooth,variable=\x,blue] plot ({\x},{(10/729)*(1 + sin(pi*(\x*(729/10) - (5 + 1/2)) r))/2 + 2*10/729});

\draw[domain=7*10/729:7.5*10/729,smooth,variable=\x,blue] plot ({\x},{(3*10/729)*(1 + sin(2*pi*(\x*(729/10) + 1/4 - 7) r))/2 + 0*10/729});

\draw[domain=7.5*10/729:8*10/729,smooth,variable=\x,blue] plot ({\x},{0});

\begin{scope}[shift={(2*10/81,1*10/81)}]
\draw (0,0) -- (10/729,0) -- (10/729, 10/729) -- (0,10/729) -- (0,0);
\draw (0 + 2*10/729, 0 + 10/729) -- (10/729 + 2*10/729, 0 + 10/729) -- (10/729 + 2*10/729, 10/729 + 10/729) -- (0 + 2*10/729, 10/729 + 10/729)  -- (0 + 2*10/729, 0 + 10/729);
\draw (0 + 6*10/729, 0 + 3*10/729) -- (10/729 + 6*10/729, 0 + 3*10/729) -- (10/729 + 6*10/729, 10/729 + 3*10/729) -- (0 + 6*10/729, 10/729 + 3*10/729)  -- (0 + 6*10/729, 0 + 3*10/729);

\draw (0 + 8*10/729, 0 + 0*10/729) -- (10/729 + 8*10/729, 0 + 0*10/729) -- (10/729 + 8*10/729, 10/729 + 0*10/729) -- (0 + 8*10/729, 10/729 + 0*10/729)  -- (0 + 8*10/729, 0 + 0*10/729);

\draw[loosely dotted] (4*10/729 + 0/162, 2*10/729 + 0/324) -- (4*10/729 + 5/162, 2*10/729 + 5/324);

\draw[domain=10/729:2*10/729,smooth,variable=\x,blue] plot ({\x},{(10/729)*(1 + sin(pi*(\x*(729/10) - 1 - 1/2) r))/2});

\draw[domain=3*10/729:4*10/729,smooth,variable=\x,blue] plot ({\x},{(10/729)*(1 + sin(pi*(\x*(729/10) - (3 + 1/2)) r))/2 + 1*10/729});

\draw[domain=5.7*10/729:6*10/729,smooth,variable=\x,blue] plot ({\x},{(10/729)*(1 + sin(pi*(\x*(729/10) - (5 + 1/2)) r))/2 + 2*10/729});

\draw[domain=7*10/729:7.5*10/729,smooth,variable=\x,blue] plot ({\x},{(3*10/729)*(1 + sin(2*pi*(\x*(729/10) + 1/4 - 7) r))/2 + 0*10/729});

\draw[domain=7.5*10/729:8*10/729,smooth,variable=\x,blue] plot ({\x},{0});

\end{scope}

\begin{scope}[shift={(6*10/81,3*10/81)}]
\draw (0,0) -- (10/729,0) -- (10/729, 10/729) -- (0,10/729) -- (0,0);
\draw (0 + 2*10/729, 0 + 10/729) -- (10/729 + 2*10/729, 0 + 10/729) -- (10/729 + 2*10/729, 10/729 + 10/729) -- (0 + 2*10/729, 10/729 + 10/729)  -- (0 + 2*10/729, 0 + 10/729);
\draw (0 + 6*10/729, 0 + 3*10/729) -- (10/729 + 6*10/729, 0 + 3*10/729) -- (10/729 + 6*10/729, 10/729 + 3*10/729) -- (0 + 6*10/729, 10/729 + 3*10/729)  -- (0 + 6*10/729, 0 + 3*10/729);

\draw (0 + 8*10/729, 0 + 0*10/729) -- (10/729 + 8*10/729, 0 + 0*10/729) -- (10/729 + 8*10/729, 10/729 + 0*10/729) -- (0 + 8*10/729, 10/729 + 0*10/729)  -- (0 + 8*10/729, 0 + 0*10/729);

\draw[loosely dotted] (4*10/729 + 0/162, 2*10/729 + 0/324) -- (4*10/729 + 5/162, 2*10/729 + 5/324);

\draw[domain=10/729:2*10/729,smooth,variable=\x,blue] plot ({\x},{(10/729)*(1 + sin(pi*(\x*(729/10) - 1 - 1/2) r))/2});

\draw[domain=3*10/729:4*10/729,smooth,variable=\x,blue] plot ({\x},{(10/729)*(1 + sin(pi*(\x*(729/10) - (3 + 1/2)) r))/2 + 1*10/729});

\draw[domain=5.7*10/729:6*10/729,smooth,variable=\x,blue] plot ({\x},{(10/729)*(1 + sin(pi*(\x*(729/10) - (5 + 1/2)) r))/2 + 2*10/729});

\draw[domain=7*10/729:7.5*10/729,smooth,variable=\x,blue] plot ({\x},{(3*10/729)*(1 + sin(2*pi*(\x*(729/10) + 1/4 - 7) r))/2 + 0*10/729});

\draw[domain=7.5*10/729:8*10/729,smooth,variable=\x,blue] plot ({\x},{0});
\end{scope}

\begin{scope}[shift={(8*10/81,0*10/81)}]
\draw (0,0) -- (10/729,0) -- (10/729, 10/729) -- (0,10/729) -- (0,0);
\draw (0 + 2*10/729, 0 + 10/729) -- (10/729 + 2*10/729, 0 + 10/729) -- (10/729 + 2*10/729, 10/729 + 10/729) -- (0 + 2*10/729, 10/729 + 10/729)  -- (0 + 2*10/729, 0 + 10/729);
\draw (0 + 6*10/729, 0 + 3*10/729) -- (10/729 + 6*10/729, 0 + 3*10/729) -- (10/729 + 6*10/729, 10/729 + 3*10/729) -- (0 + 6*10/729, 10/729 + 3*10/729)  -- (0 + 6*10/729, 0 + 3*10/729);

\draw (0 + 8*10/729, 0 + 0*10/729) -- (10/729 + 8*10/729, 0 + 0*10/729) -- (10/729 + 8*10/729, 10/729 + 0*10/729) -- (0 + 8*10/729, 10/729 + 0*10/729)  -- (0 + 8*10/729, 0 + 0*10/729);

\draw[loosely dotted] (4*10/729 + 0/162, 2*10/729 + 0/324) -- (4*10/729 + 5/162, 2*10/729 + 5/324);

\draw[domain=10/729:2*10/729,smooth,variable=\x,blue] plot ({\x},{(10/729)*(1 + sin(pi*(\x*(729/10) - 1 - 1/2) r))/2});

\draw[domain=3*10/729:4*10/729,smooth,variable=\x,blue] plot ({\x},{(10/729)*(1 + sin(pi*(\x*(729/10) - (3 + 1/2)) r))/2 + 1*10/729});

\draw[domain=5.7*10/729:6*10/729,smooth,variable=\x,blue] plot ({\x},{(10/729)*(1 + sin(pi*(\x*(729/10) - (5 + 1/2)) r))/2 + 2*10/729});

\draw[domain=7*10/729:7.5*10/729,smooth,variable=\x,blue] plot ({\x},{(3*10/729)*(1 + sin(2*pi*(\x*(729/10) + 1/4 - 7) r))/2 + 0*10/729});

\draw[domain=7.5*10/729:8*10/729,smooth,variable=\x,blue] plot ({\x},{0});
\end{scope}

\end{scope}

\begin{scope}[shift={(6*10/9, 3*10/9)}]
\draw (0,0) -- (10/729,0) -- (10/729, 10/729) -- (0,10/729) -- (0,0);
\draw (0 + 2*10/729, 0 + 10/729) -- (10/729 + 2*10/729, 0 + 10/729) -- (10/729 + 2*10/729, 10/729 + 10/729) -- (0 + 2*10/729, 10/729 + 10/729)  -- (0 + 2*10/729, 0 + 10/729);
\draw (0 + 6*10/729, 0 + 3*10/729) -- (10/729 + 6*10/729, 0 + 3*10/729) -- (10/729 + 6*10/729, 10/729 + 3*10/729) -- (0 + 6*10/729, 10/729 + 3*10/729)  -- (0 + 6*10/729, 0 + 3*10/729);

\draw (0 + 8*10/729, 0 + 0*10/729) -- (10/729 + 8*10/729, 0 + 0*10/729) -- (10/729 + 8*10/729, 10/729 + 0*10/729) -- (0 + 8*10/729, 10/729 + 0*10/729)  -- (0 + 8*10/729, 0 + 0*10/729);

\draw[loosely dotted] (4*10/729 + 0/162, 2*10/729 + 0/324) -- (4*10/729 + 5/162, 2*10/729 + 5/324);

\draw[domain=10/729:2*10/729,smooth,variable=\x,blue] plot ({\x},{(10/729)*(1 + sin(pi*(\x*(729/10) - 1 - 1/2) r))/2});

\draw[domain=3*10/729:4*10/729,smooth,variable=\x,blue] plot ({\x},{(10/729)*(1 + sin(pi*(\x*(729/10) - (3 + 1/2)) r))/2 + 1*10/729});

\draw[domain=5.7*10/729:6*10/729,smooth,variable=\x,blue] plot ({\x},{(10/729)*(1 + sin(pi*(\x*(729/10) - (5 + 1/2)) r))/2 + 2*10/729});

\draw[domain=7*10/729:7.5*10/729,smooth,variable=\x,blue] plot ({\x},{(3*10/729)*(1 + sin(2*pi*(\x*(729/10) + 1/4 - 7) r))/2 + 0*10/729});

\draw[domain=7.5*10/729:8*10/729,smooth,variable=\x,blue] plot ({\x},{0});

\begin{scope}[shift={(2*10/81,1*10/81)}]
\draw (0,0) -- (10/729,0) -- (10/729, 10/729) -- (0,10/729) -- (0,0);
\draw (0 + 2*10/729, 0 + 10/729) -- (10/729 + 2*10/729, 0 + 10/729) -- (10/729 + 2*10/729, 10/729 + 10/729) -- (0 + 2*10/729, 10/729 + 10/729)  -- (0 + 2*10/729, 0 + 10/729);
\draw (0 + 6*10/729, 0 + 3*10/729) -- (10/729 + 6*10/729, 0 + 3*10/729) -- (10/729 + 6*10/729, 10/729 + 3*10/729) -- (0 + 6*10/729, 10/729 + 3*10/729)  -- (0 + 6*10/729, 0 + 3*10/729);

\draw (0 + 8*10/729, 0 + 0*10/729) -- (10/729 + 8*10/729, 0 + 0*10/729) -- (10/729 + 8*10/729, 10/729 + 0*10/729) -- (0 + 8*10/729, 10/729 + 0*10/729)  -- (0 + 8*10/729, 0 + 0*10/729);

\draw[loosely dotted] (4*10/729 + 0/162, 2*10/729 + 0/324) -- (4*10/729 + 5/162, 2*10/729 + 5/324);

\draw[domain=10/729:2*10/729,smooth,variable=\x,blue] plot ({\x},{(10/729)*(1 + sin(pi*(\x*(729/10) - 1 - 1/2) r))/2});

\draw[domain=3*10/729:4*10/729,smooth,variable=\x,blue] plot ({\x},{(10/729)*(1 + sin(pi*(\x*(729/10) - (3 + 1/2)) r))/2 + 1*10/729});

\draw[domain=5.7*10/729:6*10/729,smooth,variable=\x,blue] plot ({\x},{(10/729)*(1 + sin(pi*(\x*(729/10) - (5 + 1/2)) r))/2 + 2*10/729});

\draw[domain=7*10/729:7.5*10/729,smooth,variable=\x,blue] plot ({\x},{(3*10/729)*(1 + sin(2*pi*(\x*(729/10) + 1/4 - 7) r))/2 + 0*10/729});

\draw[domain=7.5*10/729:8*10/729,smooth,variable=\x,blue] plot ({\x},{0});

\end{scope}

\begin{scope}[shift={(6*10/81,3*10/81)}]
\draw (0,0) -- (10/729,0) -- (10/729, 10/729) -- (0,10/729) -- (0,0);
\draw (0 + 2*10/729, 0 + 10/729) -- (10/729 + 2*10/729, 0 + 10/729) -- (10/729 + 2*10/729, 10/729 + 10/729) -- (0 + 2*10/729, 10/729 + 10/729)  -- (0 + 2*10/729, 0 + 10/729);
\draw (0 + 6*10/729, 0 + 3*10/729) -- (10/729 + 6*10/729, 0 + 3*10/729) -- (10/729 + 6*10/729, 10/729 + 3*10/729) -- (0 + 6*10/729, 10/729 + 3*10/729)  -- (0 + 6*10/729, 0 + 3*10/729);

\draw (0 + 8*10/729, 0 + 0*10/729) -- (10/729 + 8*10/729, 0 + 0*10/729) -- (10/729 + 8*10/729, 10/729 + 0*10/729) -- (0 + 8*10/729, 10/729 + 0*10/729)  -- (0 + 8*10/729, 0 + 0*10/729);

\draw[loosely dotted] (4*10/729 + 0/162, 2*10/729 + 0/324) -- (4*10/729 + 5/162, 2*10/729 + 5/324);

\draw[domain=10/729:2*10/729,smooth,variable=\x,blue] plot ({\x},{(10/729)*(1 + sin(pi*(\x*(729/10) - 1 - 1/2) r))/2});

\draw[domain=3*10/729:4*10/729,smooth,variable=\x,blue] plot ({\x},{(10/729)*(1 + sin(pi*(\x*(729/10) - (3 + 1/2)) r))/2 + 1*10/729});

\draw[domain=5.7*10/729:6*10/729,smooth,variable=\x,blue] plot ({\x},{(10/729)*(1 + sin(pi*(\x*(729/10) - (5 + 1/2)) r))/2 + 2*10/729});

\draw[domain=7*10/729:7.5*10/729,smooth,variable=\x,blue] plot ({\x},{(3*10/729)*(1 + sin(2*pi*(\x*(729/10) + 1/4 - 7) r))/2 + 0*10/729});

\draw[domain=7.5*10/729:8*10/729,smooth,variable=\x,blue] plot ({\x},{0});
\end{scope}

\begin{scope}[shift={(8*10/81,0*10/81)}]
\draw (0,0) -- (10/729,0) -- (10/729, 10/729) -- (0,10/729) -- (0,0);
\draw (0 + 2*10/729, 0 + 10/729) -- (10/729 + 2*10/729, 0 + 10/729) -- (10/729 + 2*10/729, 10/729 + 10/729) -- (0 + 2*10/729, 10/729 + 10/729)  -- (0 + 2*10/729, 0 + 10/729);
\draw (0 + 6*10/729, 0 + 3*10/729) -- (10/729 + 6*10/729, 0 + 3*10/729) -- (10/729 + 6*10/729, 10/729 + 3*10/729) -- (0 + 6*10/729, 10/729 + 3*10/729)  -- (0 + 6*10/729, 0 + 3*10/729);

\draw (0 + 8*10/729, 0 + 0*10/729) -- (10/729 + 8*10/729, 0 + 0*10/729) -- (10/729 + 8*10/729, 10/729 + 0*10/729) -- (0 + 8*10/729, 10/729 + 0*10/729)  -- (0 + 8*10/729, 0 + 0*10/729);

\draw[loosely dotted] (4*10/729 + 0/162, 2*10/729 + 0/324) -- (4*10/729 + 5/162, 2*10/729 + 5/324);

\draw[domain=10/729:2*10/729,smooth,variable=\x,blue] plot ({\x},{(10/729)*(1 + sin(pi*(\x*(729/10) - 1 - 1/2) r))/2});

\draw[domain=3*10/729:4*10/729,smooth,variable=\x,blue] plot ({\x},{(10/729)*(1 + sin(pi*(\x*(729/10) - (3 + 1/2)) r))/2 + 1*10/729});

\draw[domain=5.7*10/729:6*10/729,smooth,variable=\x,blue] plot ({\x},{(10/729)*(1 + sin(pi*(\x*(729/10) - (5 + 1/2)) r))/2 + 2*10/729});

\draw[domain=7*10/729:7.5*10/729,smooth,variable=\x,blue] plot ({\x},{(3*10/729)*(1 + sin(2*pi*(\x*(729/10) + 1/4 - 7) r))/2 + 0*10/729});

\draw[domain=7.5*10/729:8*10/729,smooth,variable=\x,blue] plot ({\x},{0});
\end{scope}

\end{scope}

\begin{scope}[shift={(8*10/9, 0*10/9)}]
\draw (0,0) -- (10/729,0) -- (10/729, 10/729) -- (0,10/729) -- (0,0);
\draw (0 + 2*10/729, 0 + 10/729) -- (10/729 + 2*10/729, 0 + 10/729) -- (10/729 + 2*10/729, 10/729 + 10/729) -- (0 + 2*10/729, 10/729 + 10/729)  -- (0 + 2*10/729, 0 + 10/729);
\draw (0 + 6*10/729, 0 + 3*10/729) -- (10/729 + 6*10/729, 0 + 3*10/729) -- (10/729 + 6*10/729, 10/729 + 3*10/729) -- (0 + 6*10/729, 10/729 + 3*10/729)  -- (0 + 6*10/729, 0 + 3*10/729);

\draw (0 + 8*10/729, 0 + 0*10/729) -- (10/729 + 8*10/729, 0 + 0*10/729) -- (10/729 + 8*10/729, 10/729 + 0*10/729) -- (0 + 8*10/729, 10/729 + 0*10/729)  -- (0 + 8*10/729, 0 + 0*10/729);

\draw[loosely dotted] (4*10/729 + 0/162, 2*10/729 + 0/324) -- (4*10/729 + 5/162, 2*10/729 + 5/324);

\draw[domain=10/729:2*10/729,smooth,variable=\x,blue] plot ({\x},{(10/729)*(1 + sin(pi*(\x*(729/10) - 1 - 1/2) r))/2});

\draw[domain=3*10/729:4*10/729,smooth,variable=\x,blue] plot ({\x},{(10/729)*(1 + sin(pi*(\x*(729/10) - (3 + 1/2)) r))/2 + 1*10/729});

\draw[domain=5.7*10/729:6*10/729,smooth,variable=\x,blue] plot ({\x},{(10/729)*(1 + sin(pi*(\x*(729/10) - (5 + 1/2)) r))/2 + 2*10/729});

\draw[domain=7*10/729:7.5*10/729,smooth,variable=\x,blue] plot ({\x},{(3*10/729)*(1 + sin(2*pi*(\x*(729/10) + 1/4 - 7) r))/2 + 0*10/729});

\draw[domain=7.5*10/729:8*10/729,smooth,variable=\x,blue] plot ({\x},{0});

\begin{scope}[shift={(2*10/81,1*10/81)}]
\draw (0,0) -- (10/729,0) -- (10/729, 10/729) -- (0,10/729) -- (0,0);
\draw (0 + 2*10/729, 0 + 10/729) -- (10/729 + 2*10/729, 0 + 10/729) -- (10/729 + 2*10/729, 10/729 + 10/729) -- (0 + 2*10/729, 10/729 + 10/729)  -- (0 + 2*10/729, 0 + 10/729);
\draw (0 + 6*10/729, 0 + 3*10/729) -- (10/729 + 6*10/729, 0 + 3*10/729) -- (10/729 + 6*10/729, 10/729 + 3*10/729) -- (0 + 6*10/729, 10/729 + 3*10/729)  -- (0 + 6*10/729, 0 + 3*10/729);

\draw (0 + 8*10/729, 0 + 0*10/729) -- (10/729 + 8*10/729, 0 + 0*10/729) -- (10/729 + 8*10/729, 10/729 + 0*10/729) -- (0 + 8*10/729, 10/729 + 0*10/729)  -- (0 + 8*10/729, 0 + 0*10/729);

\draw[loosely dotted] (4*10/729 + 0/162, 2*10/729 + 0/324) -- (4*10/729 + 5/162, 2*10/729 + 5/324);

\draw[domain=10/729:2*10/729,smooth,variable=\x,blue] plot ({\x},{(10/729)*(1 + sin(pi*(\x*(729/10) - 1 - 1/2) r))/2});

\draw[domain=3*10/729:4*10/729,smooth,variable=\x,blue] plot ({\x},{(10/729)*(1 + sin(pi*(\x*(729/10) - (3 + 1/2)) r))/2 + 1*10/729});

\draw[domain=5.7*10/729:6*10/729,smooth,variable=\x,blue] plot ({\x},{(10/729)*(1 + sin(pi*(\x*(729/10) - (5 + 1/2)) r))/2 + 2*10/729});

\draw[domain=7*10/729:7.5*10/729,smooth,variable=\x,blue] plot ({\x},{(3*10/729)*(1 + sin(2*pi*(\x*(729/10) + 1/4 - 7) r))/2 + 0*10/729});

\draw[domain=7.5*10/729:8*10/729,smooth,variable=\x,blue] plot ({\x},{0});

\end{scope}

\begin{scope}[shift={(6*10/81,3*10/81)}]
\draw (0,0) -- (10/729,0) -- (10/729, 10/729) -- (0,10/729) -- (0,0);
\draw (0 + 2*10/729, 0 + 10/729) -- (10/729 + 2*10/729, 0 + 10/729) -- (10/729 + 2*10/729, 10/729 + 10/729) -- (0 + 2*10/729, 10/729 + 10/729)  -- (0 + 2*10/729, 0 + 10/729);
\draw (0 + 6*10/729, 0 + 3*10/729) -- (10/729 + 6*10/729, 0 + 3*10/729) -- (10/729 + 6*10/729, 10/729 + 3*10/729) -- (0 + 6*10/729, 10/729 + 3*10/729)  -- (0 + 6*10/729, 0 + 3*10/729);

\draw (0 + 8*10/729, 0 + 0*10/729) -- (10/729 + 8*10/729, 0 + 0*10/729) -- (10/729 + 8*10/729, 10/729 + 0*10/729) -- (0 + 8*10/729, 10/729 + 0*10/729)  -- (0 + 8*10/729, 0 + 0*10/729);

\draw[loosely dotted] (4*10/729 + 0/162, 2*10/729 + 0/324) -- (4*10/729 + 5/162, 2*10/729 + 5/324);

\draw[domain=10/729:2*10/729,smooth,variable=\x,blue] plot ({\x},{(10/729)*(1 + sin(pi*(\x*(729/10) - 1 - 1/2) r))/2});

\draw[domain=3*10/729:4*10/729,smooth,variable=\x,blue] plot ({\x},{(10/729)*(1 + sin(pi*(\x*(729/10) - (3 + 1/2)) r))/2 + 1*10/729});

\draw[domain=5.7*10/729:6*10/729,smooth,variable=\x,blue] plot ({\x},{(10/729)*(1 + sin(pi*(\x*(729/10) - (5 + 1/2)) r))/2 + 2*10/729});

\draw[domain=7*10/729:7.5*10/729,smooth,variable=\x,blue] plot ({\x},{(3*10/729)*(1 + sin(2*pi*(\x*(729/10) + 1/4 - 7) r))/2 + 0*10/729});

\draw[domain=7.5*10/729:8*10/729,smooth,variable=\x,blue] plot ({\x},{0});
\end{scope}

\begin{scope}[shift={(8*10/81,0*10/81)}]
\draw (0,0) -- (10/729,0) -- (10/729, 10/729) -- (0,10/729) -- (0,0);
\draw (0 + 2*10/729, 0 + 10/729) -- (10/729 + 2*10/729, 0 + 10/729) -- (10/729 + 2*10/729, 10/729 + 10/729) -- (0 + 2*10/729, 10/729 + 10/729)  -- (0 + 2*10/729, 0 + 10/729);
\draw (0 + 6*10/729, 0 + 3*10/729) -- (10/729 + 6*10/729, 0 + 3*10/729) -- (10/729 + 6*10/729, 10/729 + 3*10/729) -- (0 + 6*10/729, 10/729 + 3*10/729)  -- (0 + 6*10/729, 0 + 3*10/729);

\draw (0 + 8*10/729, 0 + 0*10/729) -- (10/729 + 8*10/729, 0 + 0*10/729) -- (10/729 + 8*10/729, 10/729 + 0*10/729) -- (0 + 8*10/729, 10/729 + 0*10/729)  -- (0 + 8*10/729, 0 + 0*10/729);

\draw[loosely dotted] (4*10/729 + 0/162, 2*10/729 + 0/324) -- (4*10/729 + 5/162, 2*10/729 + 5/324);

\draw[domain=10/729:2*10/729,smooth,variable=\x,blue] plot ({\x},{(10/729)*(1 + sin(pi*(\x*(729/10) - 1 - 1/2) r))/2});

\draw[domain=3*10/729:4*10/729,smooth,variable=\x,blue] plot ({\x},{(10/729)*(1 + sin(pi*(\x*(729/10) - (3 + 1/2)) r))/2 + 1*10/729});

\draw[domain=5.7*10/729:6*10/729,smooth,variable=\x,blue] plot ({\x},{(10/729)*(1 + sin(pi*(\x*(729/10) - (5 + 1/2)) r))/2 + 2*10/729});

\draw[domain=7*10/729:7.5*10/729,smooth,variable=\x,blue] plot ({\x},{(3*10/729)*(1 + sin(2*pi*(\x*(729/10) + 1/4 - 7) r))/2 + 0*10/729});

\draw[domain=7.5*10/729:8*10/729,smooth,variable=\x,blue] plot ({\x},{0});
\end{scope}

\end{scope}

\end{tikzpicture}
\end{center}

\*

\noindent Let $k \in \mathbb{Z}_{\geq 0}$ and $\beta < 1$.

\*

Let $n$ refer to the level of iteration we are considering at a given time.

\*

Let $b$ be the number of boxes in the initial iteration level ($n = 1$), and let $m$ be the total number of solid curves and boxes in the initial iteration. We shall choose $b = \frac{m + 1}{2}$.

Note: This forces $m$ to be an odd natural number.

\*

{\bf Construction at iteration level $n = 1$.} We begin with $b$ boxes of dimension $\frac{1}{m} \times \frac{\beta}{m}$ arranged in the unit square so that the first $b - 1$ boxes form a diagonal with bottom left corners having coordinates $\left(\frac{2i}{m},\frac{i}{m} \right)$, for $0 \leq i \leq b-2$. The remaining box then has its bottom left corner placed at $\left( \frac{2b - 2}{m}, 0 \right)$.

\vfill

\pagebreak

\noindent To connect the first $b - 1$ boxes we use smooth curves beginning at the bottom right-hand corner of one box and ending at the bottom left-hand corner of the next box. We choose these curves, $g_{k1} (x)$, to be translations of the solution to
\begin{eqnarray}
\frac{d g_{k1}}{dx} = C (mx)^{k} (1 - mx)^{k}, \hspace{0.5in} g_{k1} (0) = 0, \hspace{0.5in} g_{k1} \left( \frac{1}{m} \right) = \frac{\beta}{m}, \nonumber
\end{eqnarray}
on the interval $\left[ 0,\frac{1}{m} \right]$, for some constant $C$. This constant is given in \cite{slat}.

Note: Any suitable flat function would work here, all we require is a $C^k$ function on a closed interval with trivial first $k$-derivatives at both ends.

Solving the above ODE gives us the following connecting curves
\begin{eqnarray}
g_{k1} (x) & = & \frac{\beta}{m} \frac{1}{\sum_{i = 0}^{k} \binom{k}{i} \frac{(-1)^i}{k + 1 + i}} \sum_{i=0}^{k} \binom{k}{i} \frac{(-1)^{i}}{k + 1 + i} (mx)^{k + 1 + i} \nonumber \\
& = & \frac{\beta}{m} \frac{\left(2k + 1 \right)!}{\left(k! \right)^2} \sum_{i=0}^{k} \binom{k}{i} \frac{(-1)^{i}}{k + 1 + i} (mx)^{k +1 + i}.\nonumber
\end{eqnarray}

\noindent To join the penultimate box to the final box we use a translation of the previous curve combined with a reflection and scaling: 
\begin{eqnarray}
h_{k1} (x) & = & \frac{\beta}{m} (b - 2) \frac{(2k+1)!}{\left(k! \right)^2} \sum_{i=0}^{k} \binom{k}{i} \frac{(-1)^{i}}{k +1 + i} \left( 1 - 2 mx \right)^{k +1 + i}, \nonumber
\end{eqnarray}
for $0 \leq x \leq \frac{1}{m}$, and $h_{k1}(x) = 0$ for $\frac{1}{2m} \leq x \leq \frac{1}{m}$.

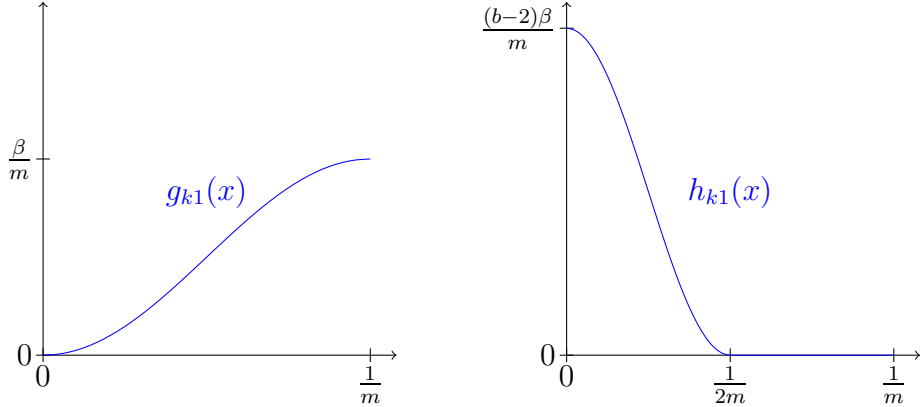
\begin{center}
\begin{tikzpicture}[scale=0.87]
\draw[->] (0,0) -- (0,5.4);
\draw[->] (0,0) -- (5.4,0);

\foreach \x in {0, 5}
\draw[shift={(\x,0)},color=black] (0pt,3pt) -- (0pt,-3pt);
\draw (0,0) node[below] {$0$};
\draw (5,0) node[below] {$\frac{1}{m}$};

\foreach \y in {0, 3}
\draw[shift={(0,\y)},color=black] (-3pt,0pt) -- (3pt,0pt);
\draw (0,0) node[left] {$0$};
\draw (0,3) node[left] {$\frac{\beta}{m}$};

\draw[domain=0:5,smooth,variable=\x,blue] plot ({\x},{(3/2)*(1 - cos(pi*\x/5 r))});

\draw (2.5,2.5) node[color=blue] {$g_{k1}(x)$};

\draw[->] (8,0) -- (8,5.4);
\draw[->] (8,0) -- (8 + 5.4,0);

\foreach \x in {0,2.5,5}
\draw[shift={(8 + \x,0)},color=black] (0pt,3pt) -- (0pt,-3pt);
\draw (8 + 0,0) node[below] {$0$};
\draw (8 + 2.5,0) node[below] {$\frac{1}{2m}$};
\draw (8 + 5,0) node[below] {$\frac{1}{m}$};

\foreach \x in {0,5}
\draw[shift={(8,\x)},color=black] (-3pt,0pt) -- (3pt,0pt);
\draw (8,0) node[left] {$0$};
\draw (8,5) node[left] {$\frac{(b - 2) \beta}{m}$};

\draw[domain=10.5:13,variable=\x,blue] plot ({\x},{0});

\draw (10.5,2.5) node[color=blue] {$h_{k1}(x)$};

\begin{scope}[shift={(8,0)}]
\draw[domain=0:2.5,smooth,variable=\x,blue] plot ({\x},{(5/2)*(1 + cos((2*pi*\x/5) r))});
\end{scope}

\end{tikzpicture}
\end{center}
This gives us the first iteration: $n = 1$. \vspace{-0.05in}

\*

\noindent For the next iteration, $n = 2$, we take the $b$ boxes of dimension $\frac{1}{m} \times \frac{\beta}{m}$, and into each of these boxes we identically construct a new collection of boxes and curves similar to those in iteration $n = 1$, with the exception that the new boxes have dimension $\frac{1}{m^2} \times \frac{\beta^{2}}{m^2}$ and the new curves are all appropriately scaled so that they are all translations of
\begin{eqnarray}
g_{k2}(x) = \frac{\beta}{m} g_{k1} \left( mx \right) \hspace{0.2in} \text{ and } \hspace{0.2in} h_{k2}(x) = \frac{\beta}{m} h_{k1} \left( mx \right). \nonumber
\end{eqnarray}

We then repeat this process ad infinitum, for each iteration $n$. \vspace{-0.05in}

\*

This gives us our function $f(x) : [0,1] \rightarrow [0,1]$.

\pagebreak

\begin{claim}
$f(x) \in C^{k}\left( [0,1] \right)$.
\end{claim}
\begin{proof}
\noindent The domain of $f(x)$ can be broken in to two groups: interior points on which the solid curves are defined and boundary points at the left and right endpoints of some box.

\vspace{0.1in}

\noindent It is clear that $f \in C^{k}$ for any interior point on which a solid curve is defined. It remains to establish that $f \in C^{k}$ at the endpoints of the boxes. More specifically, it remains to establish that $f(x)$ is $k$-times differentiable from the left for right-hand endpoints, and from the right for left-hand endpoints. We prove this by induction on order of differentiation $1 \leq j \leq k$.

\vspace{0.1in}

\noindent Let $x = e$ be any left endpoint of some box from our construction process.

\vspace{0.2in}

\noindent {\bf Case: $j=1$.}

\vspace{0.0in}

\noindent Let $\left( e_i \right)_{i \in \mathbb{N}}$ be any sequence of points, for which we have defined right-hand derivatives, that converge from the right to $e$. 
\vspace{0.1in}

\noindent By construction, for any given $i \in \mathbb{N}$ there exists $N(i) \in \mathbb{N}$ that tells us the level of the iterative process at which $f(e_i)$ was defined. Since $\displaystyle\lim_{i \to \infty} e_i = e$ it follows that $\displaystyle\lim_{i \to \infty} N(i) = \infty$.

\vspace{0.1in}

\noindent If $f(e_i)$ is defined in the $N(i)$-th level of the iterative process then
$$
\left( \frac{1}{m} \right)^{N(i)} \leq \left| e - e_i \right| < \left( \frac{1}{m} \right)^{N(i) - 1} \hspace{0.1in} \text{ and } \hspace{0.25in}
\left| f(e) - f(e_i) \right| < \left( \frac{\beta}{m} \right)^{N(i) - 1}.
$$

\vspace{0.1in}

\noindent Hence
\begin{eqnarray}
\frac{| f(e) - f(e_i) |}{|e - e_i|} < \beta^{N(i) - 1} m. \nonumber
\end{eqnarray}

\vspace{0.1in}

\noindent By definition, $\beta < 1$, and thus
\begin{eqnarray}
\partial_{+} f(e) = \lim_{i \to \infty} \frac{| f(e) - f(e_i) |}{| e - e_{i} |} \leq \lim_{i \to \infty} \beta^{N(i) - 1} m = 0 = \partial_{-} f(e), \nonumber
\end{eqnarray}
the last equality comes from our choice of the solid curves.

\vspace{0.1in}

\noindent This argument is virtually identical for right endpoints. Therefore $f(x) \in C^{1}$ and  $f^{(1)}(e) = 0$.

\*

\noindent {\bf Case $j = l \leq k$.}

\vspace{0.1in}

\noindent Assume that $f^{(1)}(e) = \cdots = f^{(l-1)}(e) = 0$ for some left endpoint, $e$, of a box. Again, let $\left( e_i \right)_{i \in \mathbb{N}}$ be any sequence of points, for which we have defined right-hand derivatives, that converge from the right to $e$.

\*

As above, there exists $N(i) \in \mathbb{N}$ telling us the level of the iterative process at which $f(e_i)$ is defined.

\*

Consider $\left| f^{(l-1)}(e) - f^{(l-1)} (e_i) \right| = \left| f^{(l-1)}(e_i) \right|$. When defining the solid curve on $x = e_{i}$ we used a translation of one of the polynomials $g_{kN(i)} (x)$ or $h_{kN(i)}(x)$. Thus $$\left| f^{(l-1)}(e) - f^{(l-1)} (e_i) \right| = \left| f^{(l-1)}(e_i) \right| = \left| g_{kN(i)}^{(l-1)} (x) \right| \text{ or } \left| h_{kN(i)}^{(l-1)} (x) \right|.$$

\*

In our construction we chose that
\begin{eqnarray}
g_{kN(i)}^{(1)} (x) = \frac{\beta^{N(i)}}{m^{N(i)}} \frac{\left(2k + 1 \right)!}{\left(k! \right)^2}  (m^{N(i)} x)^{k} \left( 1 - m^{N(i)} x \right)^{k} \nonumber
\end{eqnarray}
on $\left[ 0, \frac{1}{m^{N(i)}} \right]$ and
\begin{equation}
h_{kN(i)}^{(1)} (x) = \frac{d}{dx} \left( g_{kN(i)} \left( \frac{1}{m^{N(i)}} - 2x \right) \right) = -2 g^{(1)}_{kN(i)} \left( \frac{1}{m^{N(i)}} - 2x \right), \nonumber
\end{equation}
on $\left[ 0, \frac{1}{2m^{N(i)}} \right]$ and $h^{(1)}_{k N(i)} (x) = 0$ on $\left[ \frac{1}{2m^{N(i)}}, \frac{1}{m^{N(i)}} \right]$.

\*

Hence
\begin{eqnarray}
g_{kN(i)}^{(l - 1)} (x) = \frac{\beta^{N(i)}}{m^{N(i)}} (m^{N(i)} x)^{k + 2 - l} \left( 1 - m^{N(i)} x \right)^{k + 2 - l} p_{k,l -1,i}(m^{N(i)} x), \nonumber
\end{eqnarray}
for some polynomial $p_{k,l-1,i}(m^{N(i)} x)$ of order $l - 2$ defined on $\left[ 0, \frac{1}{m^{N(i)}} \right]$. Also
\begin{eqnarray}
h_{kN(i)}^{(l - 1)} (x) = \left(-2 \right)^{(l-1)} g^{(l-1)}_{kN(i)} \left( \frac{1}{m^{N(i)}} - 2x \right), \nonumber
\end{eqnarray}
on $\left[ 0, \frac{1}{2m^{N(i)}} \right]$ and $h^{(l-1)}_{kN(i)}(x) = 0$ on $\left[ \frac{1}{2m^{N(i)}}, \frac{1}{m^{N(i)}} \right]$.

\*

This tells us three things:
\begin{itemize}
\item[1.] The first $k$ right-derivatives of the solid curves at their left end-points are equally $0$,

\item[2.] The first $k$ left-derivatives of the solid curves at their right end-points are equally $0$,

\item[3.] Since $p_{k,l -1,i}(m^{N(i)} x)$ is a polynomial defined on $\left[ 0, \frac{1}{m^{N(i)}} \right]$ it must be bounded by some constant $c(k,l)$ only depending on $k$ and $l$. Therefore
\begin{eqnarray}
\left| f^{(l-1)}(e) - f^{(l-1)} (e_i) \right| & \leq & \max \left\{ \left| g_{kN(i)}^{(l-1)} (x) \right|, \left| h_{kN(i)}^{(l-1)} (x) \right| \right\} \nonumber \\
& \leq & \frac{\beta^{N(i)}}{m^{N(i)}} C(k,l), \nonumber
\end{eqnarray}
where $C(k,l)$ is some constant depending on $k$ and $l$.
\end{itemize}

\*

Now, as in the initial case, we have that if $f(e_i)$ is defined in the $N(i)$-th level of the iterative process then
$$
\left( \frac{1}{m} \right)^{N(i)} \leq \left| e - e_i \right| < \left( \frac{1}{m} \right)^{N(i) - 1} \hspace{0in} \text{ and } \hspace{0.2in} \left| f^{(l-1)}(e) - f^{(l-1)}(e_i) \right| < \left( \frac{\beta}{m} \right)^{N(i)} C(k,l).
$$

\*

Hence
\begin{eqnarray}
\frac{\left| f^{(l-1)}(e) - f^{(l-1)}(e_i) \right|}{\left| e - e_i \right|} \leq \beta^{N(i)} C(k,l). \nonumber
\end{eqnarray}

Taking the limit as $i \to \infty$:
\begin{eqnarray}
\partial_{+} f^{(l-1)}(e) = \lim_{i \to \infty} \frac{| f(e) - f(e_i) |}{| e - e_{i} |} \leq \lim_{i \to \infty} \beta^{N(i)} C(k,l) = 0 = \partial_{-} f^{(l-1)}(e). \nonumber
\end{eqnarray}

The argument is virtually identical for right endpoints. Thus $f(x) \in C^{l}$. This gives us the inductive step.

\*

Therefore, by strong induction, $f(x) \in C^{k} \left( [0,1] \right)$.

\end{proof}

\*

\begin{claim}
$\dim_{H} I_{1} = \frac{\log(b-1)}{\log(2b +1) - \log(\beta)}$.
\end{claim}

\begin{proof}
In each level of the iteration we added flat sections of curves. These flat sections mean that $f(x)$ has points in its range whose pre-images have Hausdorff dimension $1$.

We want to calculate the Hausdorff dimension of the collection of all these points in the range of $f(x)$, which is equivalent to calculating the Hausdorff dimension of the intersection of all the boxes in the range. Let us denote this set by $S$.

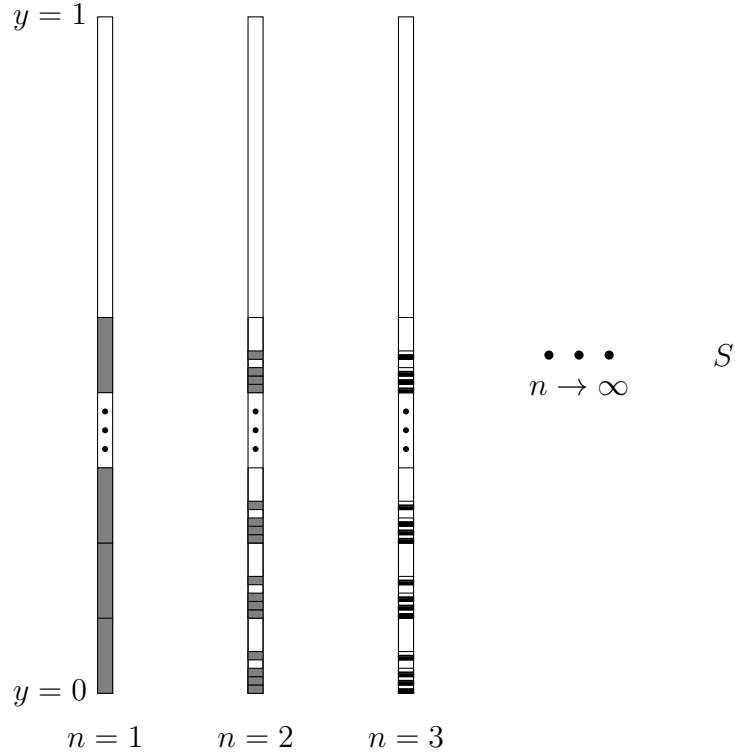
\begin{center}
\begin{tikzpicture}
\draw (0,0) rectangle (0.2,9);
\draw  (0.1,3.25) node[circle, fill, inner sep=-0.5pt] {.};
\draw  (0.1,3.5) node[circle, fill, inner sep=-0.5pt] {.};
\draw  (0.1,3.75) node[circle, fill, inner sep=-0.5pt] {.};
\foreach \y in {1,2,3,5}{
	\draw[fill=gray] (0,\y-1) rectangle (0.2,\y);
}
\draw (0.1,-0.3) node[below] {$n=1$};

\draw (0,0) node[left] {$y = 0$};
\draw (0,9) node[left] {$y = 1$};
\draw (8.6,4.5) node[left] {$S$};

\draw (2,0) rectangle (2.2,9);
\draw  (2.1,3.25) node[circle, fill, inner sep=-0.5pt] {.};
\draw  (2.1,3.5) node[circle, fill, inner sep=-0.5pt] {.};
\draw  (2.1,3.75) node[circle, fill, inner sep=-0.5pt] {.};
\foreach \l in {1,2,3,5}{
	\begin{scope}[shift={(2,\l-1)}, yscale=1/9]
	\draw (0,0) rectangle (0.2,9);
	\foreach \y in {1,2,3,5}{
		\draw[fill=gray] (0,\y-1) rectangle (0.2,\y);
	}
	\end{scope}
}
\draw (2.1,-0.3) node[below] {$n=2$};

\draw (4,0) rectangle (4.2,9);
\draw  (4.1,3.25) node[circle, fill, inner sep=-0.5pt] {.};
\draw  (4.1,3.5) node[circle, fill, inner sep=-0.5pt] {.};
\draw  (4.1,3.75) node[circle, fill, inner sep=-0.5pt] {.};
\foreach \i in {1,2,3,5}{
	\begin{scope}[shift={(4,\i-1)}, yscale=1/9]
		\draw (0,0) rectangle (0.2,9);
		\foreach \j in {1,2,3,5}{
			\begin{scope}[shift={(0,\j-1)}, yscale=1/9]
			\draw (0,0) rectangle (0.2,9);
			\foreach \y in {1,2,3,5}{
				\draw[fill=gray] (0,\y-1) rectangle (0.2,\y);
			}
			\end{scope}
		}
	\end{scope}
}
\draw (4.1,-0.3) node[below] {$n=3$};

\draw  (6,4.5) node[circle, fill, inner sep=0pt] {.};
\draw  (6.4,4.5) node[circle, fill, inner sep=0pt] {.};
\draw  (6.8,4.5) node[circle, fill, inner sep=0pt] {.};

\draw  (6.4,4.3) node[below] {$n \to \infty$};
\end{tikzpicture}
\end{center}

Set $d = \frac{\log (b -1)}{\log(2b + 1) - \log(\beta)}$. We first prove that $\dim_{H} (S) \leq d$. Suppose $\gamma > d$. The iterative process used to construct $f(x)$ gives us a sequence of coverings of $S$. At level $n = 1$ we can cover $S$ by $b - 1$ intervals of length $\frac{\beta}{m}$. At level $n = 2$ we can cover $S$ by $(b - 1)^2$ intervals of length $\left( \frac{\beta}{m} \right)^2$. After $n$ iterations we can cover $S$ by $(b - 1)^n$ intervals of length $\left( \frac{\beta}{m} \right)^n$. The $\gamma$-total length of the $n$-th cover of $S$ is then $(b - 1)^n \left( \frac{\beta}{m} \right)^{\gamma n}$.

If we take the limit of this as $n \to \infty$ we get
\begin{eqnarray}
\lim_{n \to \infty} (b - 1)^n \left( \frac{\beta}{m} \right)^{\gamma n} & = & \lim_{n \to \infty} \exp \left[ n \left( \log(b - 1) - \gamma \left( \log (m) - \log (\beta) \right) \right) \right] = 0. \nonumber
\end{eqnarray}

Therefore $C^{\gamma}_{H} (S) = 0$ and $\dim_{H} (S) \leq d = \frac{\log (b -1)}{\log(2b + 1) - \log(\beta)}$.

\*

For the other direction we will show that $C^{d}_{H} (S) > 0$.

\*

Let $(S_i)_{i \in \mathbb{N}}$ be a countable cover of $S$.

\*

By compactness \cite{steen}, given any $\varepsilon > 0$, there exist a finite collection of open intervals $D_1, \hdots, D_l$ such that $\cup_{i=1}^{\infty} S_{i} \subseteq \cup_{j=1}^{l} D_{j}$
and
$$ \sum_{j=1}^{l} \left| D_j \right|^{\alpha} < \sum_{i=1}^{\infty} \left| S_{i} \right|^{\alpha} + \varepsilon.$$

Let us choose $n$ such that $$ \left( \frac{\beta}{m} \right)^n \leq \min \left\{ | D_{j} | \ : \ j=1, \hdots, l \right\}.$$

For $i = 1, \hdots, n$ define
$$ M_{i} = \# \left\{ D_{j} \ : \ \left(\frac{\beta}{m} \right)^{i}  \leq | D_{j} | < \left( \frac{\beta}{m} \right)^{i - 1} \right\}.$$

It follows that $$ \sum_{j=1}^{l} |D_{j} |^{\alpha} \geq \sum_{j=1}^{n} M_{j} \left( \frac{\beta}{m} \right)^{j \alpha} = \sum_{j=1}^{n} M_{j} \left( \frac{1}{b - 1} \right)^{j}.$$

\*

Consider any $D_{j}$. There must exist some $i$ such that $\left( \frac{\beta}{m} \right)^{i} \leq \left| D_{j} \right| < \left(\frac{\beta}{m} \right)^{i - 1}$. Thus $D_{j}$ can intersect at most $2$ of the $(b - 1)^{i}$ intervals obtained in the $i$-th level of the iterative process. Each of these intervals produces $(b - 1)^{n - i}$ sub-intervals at the $n$-th level of the iterative process, hence $D_{j}$ contains at most $2(b - 1)^{n - i}$ intervals from the $n$-th level of the construction process. In total, the $n$-th step of the construction process has $(b-1)^{n}$ intervals. Therefore
$$ (b - 1)^{n} \leq  \sum_{i=1}^{l} 2 M_{i} (b - 1)^{n - i} \hspace{0.2in} \Rightarrow \hspace{0.2in} \frac{1}{2} \leq \sum_{i=1}^{l} \frac{M_i}{(b - 1)^{i}}. $$

Combining this with the above equation gives:
\begin{eqnarray}
\frac{1}{2} \leq \sum_{j=1}^{l} \left| D_{j} \right|^{d} < \sum_{i=1}^{\infty} \left| S_{i} \right|^{d} + \varepsilon. \nonumber
\end{eqnarray}

Let $\varepsilon = \frac{1}{4}$. Then
\begin{eqnarray}
\frac{1}{4} <  \sum_{i=1}^{\infty} \left| S_{i} \right|^{d}. \nonumber
\end{eqnarray}

Therefore $ \sum_{i=1}^{\infty} \left| S_{i} \right|^{d}$ is bounded below and hence
\begin{eqnarray}
\dim_{H} (S) \geq d = \frac{\log(b - 1)}{\log(2b+1) - \log(\beta)}. \nonumber
\end{eqnarray}

\end{proof}

Using the previous claim and letting $b \to \infty$, L'H\^{o}pital's Rule tells us that:
\begin{eqnarray}
\lim_{b \to \infty} \frac{\log(b - 1)}{\log(2b+1) - \log(\beta)} \  = \  \lim_{b \to \infty} \left[ \frac{\frac{1}{b - 1}}{\frac{2}{2b + 1}} \right] \ = \ \lim_{b \to \infty} \left[ 1 + \frac{3}{2b - 2} \right] = 1. \nonumber
\end{eqnarray}

\*

Therefore $1$ is indeed a sharp bound for $\dim_{H} I_{1}(f) \leq 1$.

\end{ex}

\end{document}